\documentclass{birkjour}
\usepackage{subfigure}
\usepackage{verbatim}
 \newtheorem{thm}{Theorem}[section]
 \newtheorem{corollary}[thm]{Corollary}
 \newtheorem{lemma}[thm]{Lemma}
 \newtheorem{Proposition}[thm]{Proposition}
 \theoremstyle{definition}
 
 \theoremstyle{remark}

 \numberwithin{equation}{section}

 \newcommand{\R}{\mathbb{R}}

   \newcommand{\U}{\mathcal{U}}
    \newcommand{\V}{\mathcal{V}}

\begin{document}

%
%

\title[ Involutes of Polygons of Constant Width in Minkowski Planes]
 {Involutes of Polygons of Constant Width in Minkowski Planes}

\author[M.Craizer]{Marcos Craizer}

\address{%
Departamento de Matem\'{a}tica- PUC-Rio\br
Rio de Janeiro\br
BRAZIL}
\email{craizer@puc-rio.br}

\author[H.Martini]{Horst Martini}

\address{%
Faculty of Mathematics\br
University of Technology\br
09107 Chemnitz\br
GERMANY}
\email{martini@mathematik.tu-chemnitz.de}

\thanks{The first named author wants to thank CNPq for financial support during the preparation of this manuscript. \newline E-mail of the corresponding author: craizer@puc-rio.br}



\date{June 18, 2015}

\begin{abstract}
Consider a convex  polygon $P$ in the plane, and denote by $U$ a homothetical copy of the vector sum of $P$ and $-P$.
Then the polygon $U$, as unit ball, induces a norm such that, with respect to this norm, $P$ has constant Minkowskian width.
We define notions like  Minkowskian curvature, evolutes and involutes for polygons of constant $U$-width, and we prove that many properties of the smooth case, which is already completely studied, are preserved.
The iteration of involutes generates a pair of sequences of polygons of constant width with respect to the Minkowski norm and its dual norm, respectively.
We prove that these sequences are converging to symmetric polygons with the same center, which can be regarded as a central point of the polygon $P$.
\end{abstract}

\maketitle

\section{Introduction}

A {\it Minkowski} or {\it normed plane} is a $2$-dimensional vector space with a norm. This norm is induced by its \emph{unit ball} $U$, which is a compact, convex set centered at the origin
(or, shortly, \emph{centered}). Thus, we write $(\R^2,U)$ for a Minkowski plane with unit ball $U$, whose boundary is the \emph{unit circle} of $(\R^2,U)$.
The geometry of normed planes and spaces, usually called \emph{Minkowski Geometry} (see \cite{Thompson96}, \cite{Ma-Sw-We}, and \cite{Ma-Sw}), is
strongly related to and influenced by the fields of Convexity, Banach Space Theory, Finsler Geometry and, more recently, Discrete and Computational Geometry.
The present paper can be considered as one of the possibly first contributions to Discrete Differential Geometry in the spirit of Minkowski Geometry.
The study of special types of curves in Minkowski planes is a promising subject (see the survey \cite{Ma-Wu}), and the particular case of curves of constant Minkowskian width has been studied for a long time
(see \cite{Chakerian66}, \cite{Chakerian83}, \cite{He-Ma}, and \S~2 of \cite{Ma-Sw}).
A curve $\gamma$ has constant Minkowskian width with respect to the unit ball $U$ or, shortly, \emph{constant $U$-width},
if $h(\gamma) + h(-\gamma)$ is constant with respect to the norm induced by $U$, where $h(\gamma)$ denotes the support
function of $\gamma$. Another concept from the classical theory of planar curves important for our paper is that of \emph{involutes and evolutes}; see, e.g., Chapter 5 of \cite{Gray} and,
respectively, \cite{GAS}. For natural
generalizations of involutes, which also might be extended from the Euclidean case to normed planes, we refer to \cite{So} and \cite{Ap-Mn}. And in \cite{Tanno} it is
shown how the concept of evolutes and involutes can help to construct curves
of constant width in the Euclidean plane.

In this paper, we consider convex polygons $P$ of constant Minkowskian width in a normed plane, for  short calling them \emph{CW-polygons}. If $P$ is a CW polygon, then the unit ball $U$ is necessarily a centered polygon
whose sides and diagonals are suitably parallel to corresponding sides and diagonals of $P$ (sometimes with diagonals suitably meaning also
sides; see \S\S~ 2.1 below). If, in particular, $U$ is homothetic to $P+(-P)$, then, and only then, $P$ is
of constant $U$-width in the Minkowski plane induced by $U$.

There are many results concerning \emph{smooth CW curves} in normed planes: Barbier's theorem fixing their circumference only by the diameter of the curve (cf. \cite{Petty} and \cite{Ma-Mu});
relations between curvature, evolutes, involutes, and equidistants
(see \cite{Tabach97} and, for applications of Minkowskian evolutes in computer graphics, \cite{Ait-Haddou00}); mixed areas, and the relation between the area and length
of a CW curve cut off along a diameter (see \cite{Chakerian66}, (2.1)). In this paper we prove corresponding results for \emph{CW polygons}.
We note that our results
are direct discretizations of the corresponding results for the smooth case, where the derivatives and integrals are replaced by differences and sums.
It is meant in this sense that the results of this paper can be considered as one of the first contributions to Discrete Differential Geometry in the framework of normed planes.

Among the $U$-equidistants of a smooth CW curve $\gamma$, there is a particular one called {\it central equidistant}. The central equidistant of $\gamma$ coincides with its {\it area evolute}, while the evolute of $\gamma$ coincides with its {\it center symmetry set} (see \cite{Craizer14} and \cite{Giblin08}). We show that for a CW polygon $P$ the same results hold: The central equidistant $M$ coincides with the area evolute, and the evolute $E$ coincides with the central symmetry set (see \cite{Craizer13}).
Since the equidistants of $P$ are the involutes of $E$, we shall choose the central equidistant as a representative of them, and we write $M=Inv(E)$.

For a Minkowski plane whose  unit ball $U$ is a centered convex (2n)-gon, the \emph{dual unit ball} $V$ is also a centered convex (2n)-gon
with diagonals parallel to the sides of $U$, and the sides parallel to diagonals of $U$. As in the smooth case (cf. \cite{Craizer14}), the involutes of the central equidistant  of $P$ form a
one-parameter family
of polygons having constant $V$-width. This one-parameter family consists of the $V$-equidistants of any of its members, and we shall choose the central equidistant $N$ as its representative.
Thus we write $N=Inv(M)$.
In \cite{Craizer14} it is proved that,  for smooth curves, the analogous 
involute
$N$ is contained in the region bounded by $M$ and has smaller or equal signed area.
In this paper we prove the corresponding fact for polygons, namely, that $N$ is contained in the region bounded by $M$ and the signed area of $N$ is not larger than the signed area of $M$.

What happens if we iterate the involutes? Let $N(0)=E$, $M(0)=M$, $N(1)=N$ and define $M(k)=Inv(N(k))$, $N(k+1)=Inv(M(k))$. Then we obtain two 
sequences
$M(k)$ and $N(k)$, the first being of constant $U$-width and the
second of constant $V$-width. Moreover, we have
$$
\overline{N(0)}\supset\overline {M(0)}\supset \overline {N(1)}\supset \overline {M(1)}\supset ...\,,
$$
where ${\overline R}$ denotes the closure of the region bounded by $R$.
Denoting by $O=O(P)$ the intersection of all these sets, we shall prove that $O$ is in fact a single point.
Another form of describing the convergence of $M(k)$ and $N(k)$ to $O$ is as follows: For fixed $c$ and $d$, consider the sequences $M(k)+cU$ of polygons of constant $U$-width, and the sequences $N(k)+dV$
of polygons of constant $V$-width. Then these sequences are converging to $O+cU$ and $O+dV$, respectively, which are $U$- and $V$-balls centered at $O$.
For smooth curves the analogous results were proved in \cite{Craizer14}.

Our paper is organized as follows: In Section 2 we describe geometrically the unit ball of a Minkowski plane for which a given convex polygon has constant Minkowskian width. In Section 3, we define Minkowskian curvature, evolutes and involutes for
CW polygons and prove many properties of them. In Section 4 we consider the involute of the central equidistant, and in Section 5 we prove that the involutes iterates are converging to a single point.

\section{Polygonal Minkowskian balls, their duals, and constant Minkowskian width}

Since faces and also width functions of convex sets behave additively under (vector or) Minkowski addition, it is clear that a polygon $P$ is of constant Minkowskian width if and only if $P+(-P)$ is a
homothetical copy of the unit ball $U$ of the respective normed plane; see, e.g., \S\S~2.3 of \cite{Ma-Sw}.
If, moreover, the homothety of $U$ and  $P + (-P)$ is only possible when $P$ itself is already centrally symmetric, then the only sets of constant $U$-width are the balls of that norm; cf., e.g., \cite{Yost}.
In the following we will have a closer look at various geometric relations between polygons $P$ of constant $U$-width and the unit ball $U$, since we
need them later.

Thus, let $P$ be an arbitrary planar convex polygon. By an abuse of notation, we shall denote by the same letter $P$ also the set of vertices of the polygon, the closed polygonal arc formed by the union of
its sides, and the convex region bounded by $P$.

\subsection{A centered polygon with parallel sides and diagonals}

Assume that $P=\{P_1,...,P_{2n}\}$ is a planar convex polygon with parallel opposite sides, i.e., the segments $P_iP_{i+1}$ and $P_{i+n}P_{i+n+1}$, $1\leq i\leq n$, are parallel.

\begin{lemma}\label{lemma:SymmetricBall}
Fix an origin $Z$ and take $U_1$ such that $U_1-Z=\frac{1}{2a}\left(P_{1}-P_{1+n}\right)$, for some $a>0$. Consider the polygon $U$ whose vertices are
\begin{equation}\label{eq:defineU}
U_i=Z+\frac{1}{2a}\left(  P_i-P_{i+n} \right),
\end{equation}
$1\leq i\leq 2n$. Then $U$ is convex, symmetric with respect to $Z$, 
$U_{i+1}-U_i\parallel P_{i+1}-P_i$ and $U_i-Z\parallel P_{i}-P_{i+n}$ for $1\leq i\leq n$ (see Figure \ref{fig:HexagonSym}). Moreover, $U$ is the unique polygon
with these properties.
\end{lemma}

\begin{figure}[htb]
 \centering
 \includegraphics[width=0.90\linewidth]{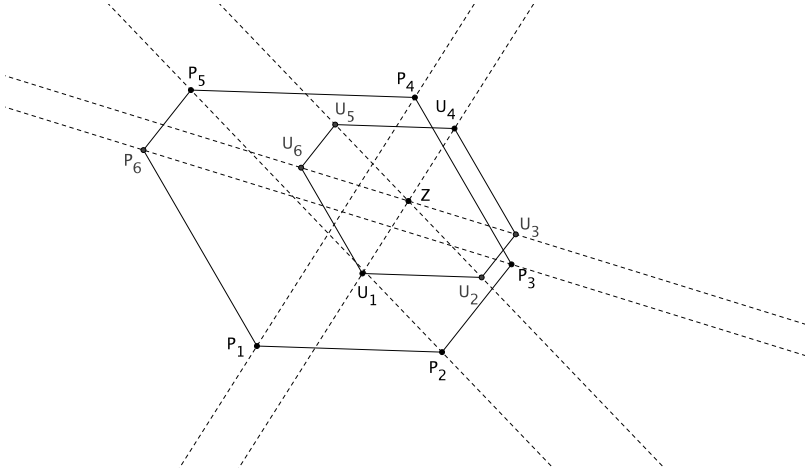}
 \caption{ A hexagon $P$ with parallel opposite sides and  the corresponding homothet $U$ of $P+(-P)$. }
\label{fig:HexagonSym}
\end{figure}

\begin{proof}
It is clear that $U$ is symmetric with respect to $Z$, $U_{i+1}-U_i\parallel P_{i+1}-P_i$ and $U_i-Z\parallel P_{i}-P_{i+n}$ for $1\leq i\leq n$. Moreover $U_{i+1}-U_i$ has the same orientation as $P_{i+1}-P_i$, which implies that $U$ is convex. 

To prove the uniqueness of $U$, observe that the point $U_2$ is obtained as the intersection of the lines parallel to $P_1P_2$ through $U_1$ and parallel to $P_2P_{2+n}$ through $Z$. 
The points $U_3,...,U_n$ are obtained inductively in a similar way, while $U_{n+1},..,U_{2n}$ are reflections of $U_1,...U_n$ with respect to $Z$. 
\end{proof}

Consider now a convex polygon $\tilde{P}=\{\tilde{P}_1,...,\tilde{P}_k\}$ that has not necessarily all opposite sides parallel. Suppose that exactly $0\leq j\leq \frac{k}{2}$ pairs are parallel. Our next 
lemma shows that the list of vertices of this polygon can be re-written as  $P=\{P_1,P_2,..,P_{2n}\}$, $n=k-j$, with "parallel opposite sides" in a broader sense.

\begin{lemma}\label{lemma:Reorder}
We may re-write the list of vertices of $\tilde{P}$ as $\{P_1,P_2,..,P_{2n}\}$ such that, for each $1\leq i\leq n$, 
$P_{i}P_{i+1}$ is parallel to $P_{i+n}P_{i+n+1}$ or else one of these sides, say $P_{i+n}P_{i+n+1}$, degenerates to a point, in which case the other side $P_iP_{i+1}$ 
is not degenerated and the line through $P_{i+n}=P_{i+n+1}$ parallel to $P_iP_{i+1}$ is outside $P$ (see Figure \ref{fig:QuadrangleSym2}). 
\end{lemma}

\begin{proof}
The polygon $\tilde{P}=\{\tilde{P}_1,...,\tilde{P}_k\}$ defines exactly $n=k-j$ directions  $\theta_1,..., \theta_n$, in increasing order, in the plane. We may assume that $\tilde{P}_1\tilde{P}_2$ is in direction $\theta_1$ and define $P_1=\tilde{P}_1$, $P_2=\tilde{P}_2$. For the induction step write $P_i=\tilde{P}_{l}$. If $P_i\tilde{P}_{l+1}$ is in direction $\theta_i$, define $P_{i+1}=\tilde{P}_{l+1}$, otherwise define $P_{i+1}=\tilde{P}_l$. It is now easy to verify that the polygon $P=\{P_1,P_2,..,P_{2n}\}$ satisfies the properties of the lemma.
\end{proof}

The construction of Lemma \ref{lemma:SymmetricBall} can be applied to the polygon $P$ obtained in Lemma \ref{lemma:Reorder} 
(see Figure \ref{fig:QuadrangleSym2}). 
If, for example, $P$ is a triangle, then $P+(-P)$ is an affinely regular hexagon (see Figure \ref{fig:Triangle}). From now on, we shall assume that $Z$ coincides with the origin of $\R^2$ and that $P=\{P_1,...,P_{2n}\}$, with $P_iP_{i+1}$ parallel to $U_iU_{i+1}$.

\begin{figure}[htb]
 \centering
 \includegraphics[width=0.70\linewidth]{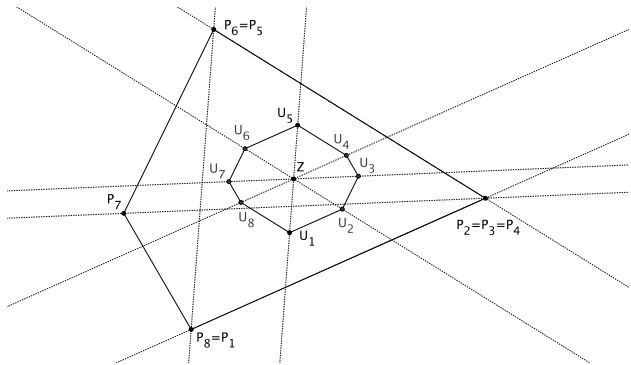}
 \caption{ A quadrangle and the corresponding symmetric octagon. }
\label{fig:QuadrangleSym2}
\end{figure}

\begin{figure}[htb]
 \centering
 \includegraphics[width=0.70\linewidth]{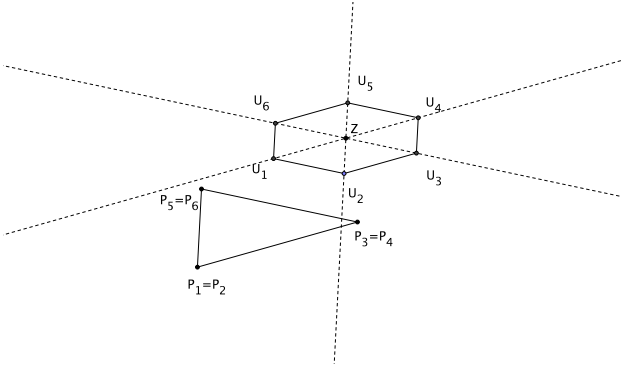}
 \caption{ When $P$ is a triangle of constant $U$-width, then $U$ is an affinely regular hexagon. }
\label{fig:Triangle}
\end{figure}

\subsection{The dual Minkowskian ball}

Now we introduce the type of duality which is very useful for our investigations.
Let $(\R^2)^*$ denote the space of linear functionals in $\R^2$. The dual norm in $(\R^2)^*$ is defined as
$$
|| f ||=\sup\{f(u), u\in U\}.
$$
We shall identify $(\R^2)^*$ with $\R^2$ by $f(\cdot)=[\cdot,v]$, where $[\cdot,\cdot]$ denotes the determinant of a pair of planar vectors.
Under this identification, the dual norm in $\R^2$ is given by
$$
|| v ||=\sup\{[u,v], u\in U\}.
$$
We shall construct below a centered polygon $V$ such that, for $v$ in any side of $V$, we have $||v||=1$. Such a polygon defines a Minkowski norm equivalent to the dual norm of $U$.

Now assume that the unit ball $U$ is a centered polygon with vertices $\{U_1,...,U_{2n}\}$, $U_{i+n}=-U_i, \ 1\leq i\leq n$. Define the polygon $V$ with vertices
\begin{equation*}
V_{i+\frac{1}{2}}=\frac{U_{i+1}-U_i}{[U_i,U_{i+1}]}.
\end{equation*}
Observe that $V_{i+n+\frac{1}{2}}=-V_{i+\frac{1}{2}}$, i.e., $V$ is centered.
Now $[V_{i+\frac{1}{2}}-V_{i-\frac{1}{2}},U_i]=0$, which implies that $V_{i+\frac{1}{2}}-V_{i-\frac{1}{2}}=-aU_i$. Multiplying both sides by $V_{i+\frac{1}{2}}$ we obtain
\begin{equation*}
U_i=-\frac{V_{i+\frac{1}{2}}-V_{i-\frac{1}{2}}}{[V_{i-\frac{1}{2}},V_{i+\frac{1}{2}}]}, 
\end{equation*}
for $1\leq i\leq 2n$.

\begin{figure}[htb]
 \centering
 \includegraphics[width=0.90\linewidth]{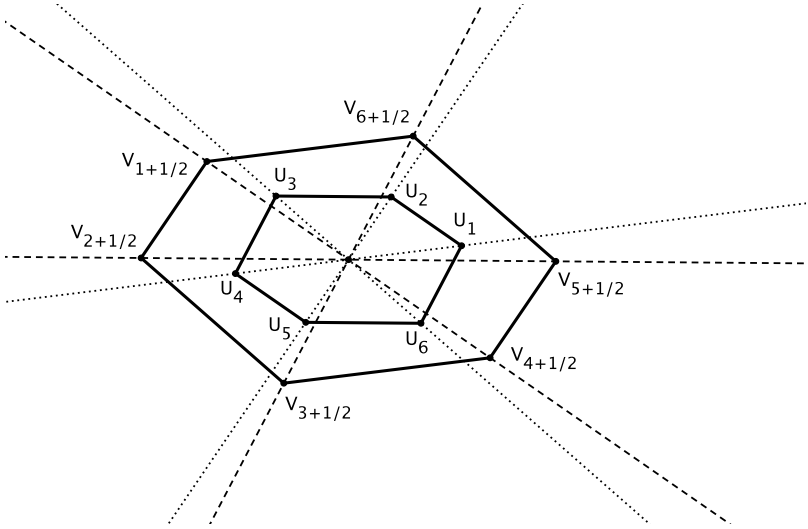}
 \caption{ The centered hexagon $U$ and its dual $V$. }
\label{fig:HexagonDuals}
\end{figure}

\begin{lemma}\label{lemma:DualBall}
The polygon $V$ is the dual unit ball.
\end{lemma}
\begin{proof}
We have that, for $1\leq i\leq 2n$,
\begin{equation}
[tU_i+(1-t)U_{i+1},V_{i+\frac{1}{2}}]=1,
\end{equation}
for any $t\in\R$ and for $j\notin \{ i,i+1\}$, $[U_j, V_{i+\frac{1}{2}}]\leq1$.  This implies that the vertex $V_{i+\frac{1}{2}}$ is from the dual unit circle.
Moreover,
\begin{equation}
[U_i,tV_{i-\frac{1}{2}}+(1-t)V_{i+\frac{1}{2}}]=1,
\end{equation}
and for $j\neq i$ we have $[U_j,tV_{i-\frac{1}{2}}+(1-t)V_{i+\frac{1}{2}}]\leq 1$, which implies that also the side $tV_{i-\frac{1}{2}}+(1-t)V_{i+\frac{1}{2}}$ is from the dual unit circle.
\end{proof}

\subsection{Polygons of constant Minkowskian width}

Consider a Minkowski plane $(\R^2,U)$, and let $P$ be a convex curve. For $f$ in the dual unit ball, the {\it support function} $h(P)(f)$ of $P$ at $f$ is defined as
\begin{equation}
h(P)(f)=\sup\{f(p), p\in P\}.
\end{equation}
The {\it width} of $P$ in the direction $f$ is defined as $w(P)(f)=h(P)(f)+h(P)(-f)$. We say that $P$ is of {\it constant Minkowskian width} if $w(P)(f)$ does not depend on $f$.

Consider now a Minkowski plane whose unit ball $U$ is a centered polygon, and let $P$ be a polygon with parallel corresponding sides and diagonals.

\begin{lemma}
In the Minkowski plane $(\R^2,U)$, $P$ has constant $U$-width.
\end{lemma}
\begin{proof}
By Lemma \ref{lemma:SymmetricBall}, we have that $P_{i}-P_{i+n}=a(U_i-U_{i+n})$, for some constant $a$. Since
$$
w(P)(V_{i+\frac{1}{2}})=h(P)(V_{i+\frac{1}{2}})+h(P)(-V_{i+\frac{1}{2}})=[P_i-P_{i+n},V_{i+\frac{1}{2}}],
$$
we obtain
$$
w(P)(V_{i+\frac{1}{2}})=2a,
$$
$1\leq i\leq 2n$, thus proving the lemma.
\end{proof}

Our next corollary says that in fact $U$ is homothetic to the Minkowski sum $P+(-P)$ (see \cite{Thompson96}, Th. 4.2.3).
\begin{corollary}
Let $P$ be a convex planar polygon and let $U$ be as in Lemma \ref{lemma:SymmetricBall}. Then $U$ is homothetic to $P+(-P)$.
\end{corollary}
\begin{proof}
We have that $2a=h(P)+h(-P)=h(P+(-P))=h(2aU)$, which implies that $P+(-P)$ is homothetic to $U$.
\end{proof}

\begin{corollary}\label{cor:CWequivalence}
Consider a centered polygon $U$ and a polygon $P$ whose sides are parallel to the corresponding sides of $U$. The following statements are equivalent:
\begin{enumerate}
\item
$P$ has constant $U$-width.
\item
$P+(-P)$ is homothetic to $U$.
\item
The corresponding diagonals of $U$ and $P$ are parallel to each other. 
\item
$P_{i}-P_{i+n}=2a(U_i-U_{i+n})$, $1\leq i\leq n$, for some constant $a$.
\end{enumerate}
\end{corollary}

\section{Geometric properties of polygons of constant Minkowskian width}

Consider a convex polygon $P=\{P_1,...,P_{2n}\}$ with parallel opposite sides and let $U=\{U_1,...,U_{2n}\}$ be the symmetric polygon obtained
from $P$ by the construction of  Lemma \ref{lemma:SymmetricBall}. 

\subsection{Central Equidistant, $V$-length, and Barbier's theorem}

\paragraph{Central equidistant} Any equidistant can be written as $P_i(c)=P_i+cU_i$, $1\leq i\leq 2n$. If we take $c=-a$, we obtain
\begin{equation}\label{eq:defineCentral}
M_i=P_i+\frac{c}{2a}\left( P_{i}-P_{i+n} \right)=\frac{1}{2}\left(P_i+P_{i+n}\right), \  1\leq i\leq 2n, 
\end{equation}
called the {\it central equidistant} of $P$. It is characterized by the condition $M_i=M_{i+n}$ (see Figure \ref{fig:OctoEqui}). If we re-scale the one-parameter family of equidistants as
\begin{equation}\label{eq:REescalaEqui}
P_i(c)=M_i+cU_i,\  1\leq i\leq 2n, 
\end{equation}
we get that the $0$-equidistant is exactly the central equidistant.

A vertex $M_i$ of the central equidistant is called a {\it cusp} if $M_{i-1}$ and $M_{i+1}$ are in the same half-plane defined by the diagonal at $P_i$.
The central equidistant coincides with the {\it area evolute} of polygons defined in \cite{Craizer13}. There 
it is proved that it has an odd number of cusps, at least three (see Figures \ref{fig:OctoEqui} and \ref{fig:OctoEvolute}).

\begin{figure}[htb]
 \centering
 \includegraphics[width=0.90\linewidth]{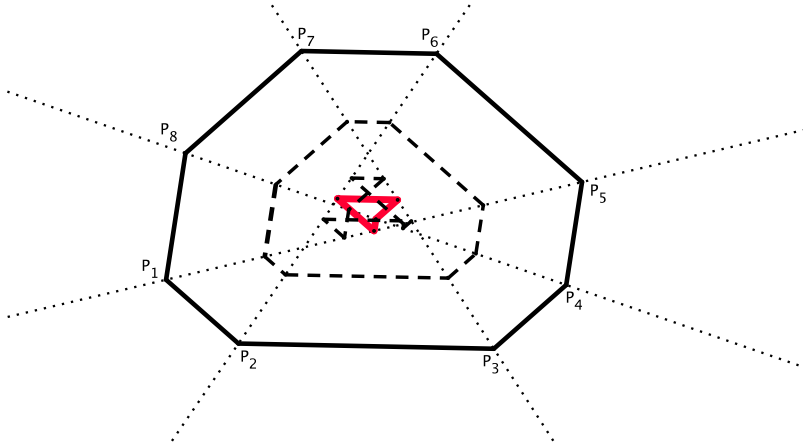}
 \caption{ The two traced octagons are ordinary equidistants. The thick quadrangle is the central equidistant. }
\label{fig:OctoEqui}
\end{figure}

\paragraph{$V$-Length}
Let $P$ be a polygonal arc whose sides are parallel to the corresponding ones of $U$. More precisely, we shall denote by
$\{P_s,...,P_t\}$ the vertices of $P$ and assume that $P_{i+1}-P_{i}$ is parallel to $V_{i+\frac{1}{2}}$.
We can write
\begin{equation}\label{eq:defineVlength}
P_{i+1}-P_i=\lambda_{i+\frac{1}{2}}V_{i+\frac{1}{2}}
\end{equation}
for some $\lambda_{i+\frac{1}{2}}\geq 0$. Then the {\it $V$-length} of the edge $P_iP_{i+1}$ is exactly $\lambda_{i+\frac{1}{2}}$, and we write
\begin{equation}\label{eq:defineVlength2}
L_V(P)=\sum_{i=s}^{t-1}\lambda_{i+\frac{1}{2}}.
\end{equation}

\paragraph{Barbier's theorem}
The classical Theorem of Barbier on curves of constant width in the Euclidean plane says that any such curve of diameter $d$ has circumference $d\pi$.
For Minkowski planes, it appears in \cite{Petty},
Th. 6.14(a), and
in \cite{Ma-Mu}. We prove here the version of this theorem for polygons.

Define $\alpha_{i+\frac{1}{2}}$, $1\leq i\leq 2n$, by the equation
\begin{equation}\label{eq:defineAlpha}
M_{i+1}-M_i=\alpha_{i+\frac{1}{2}} \left( U_{i+1}-U_i  \right)=\alpha_{i+\frac{1}{2}}[U_i,U_{i+1}]  V_{i+\frac{1}{2}}.
\end{equation}

\begin{Proposition}
Let $P(c)$ be defined by equation \eqref{eq:REescalaEqui}. Then
the $V$-length of $P(c)$ is
\begin{equation}\label{eq:Barbier}
L_V(P)=2cA(U),
\end{equation}
where $A(U)$ denotes the area of the polygon $U$.
\end{Proposition}
\begin{proof}
The $V$-length of the polygon $P(c)$ is given by
$$
L_V(P(c))=\sum_{i=1}^{2n} (\alpha_{i+\frac{1}{2}}+c) [U_i,U_{i+1}].
$$
Since $\alpha_{i+n+\frac{1}{2}}=-\alpha_{i+\frac{1}{2}}$, we obtain
$$
L_V(P(c))=c\sum_{i=1}^{2n}[U_i,U_{i+1}],
$$
which proves the proposition.
\end{proof}

If we admit signed lengths, equation \eqref{eq:Barbier} holds even for equidistants with cusps. In particular, for $c=0$ we obtain
\begin{equation}\label{eq:MlengthZero}
L_V(M)=0.
\end{equation}
For smooth closed curves this result was obtained in \cite{Tabach97} .

\subsection{Curvature and evolutes}

\paragraph{Minkowskian normals and evolutes} In the smooth case, the Minkowskian normal at a point $P$ is the line $P+sU$, where $P$ and $U$ have parallel tangents (see \cite{Tabach97}).
The evolute is the envelope of Minkowskian normals. For a polygon $P$, define the {\it Minkowskian normal} at a vertex $P_i$ as the line $P_i+sU_i$, $1\leq i\leq 2n$, and the {\it evolute} as the polygonal arc whose vertices
are the intersections of $P_i+sU_i$ and $P_{i+1}+sU_{i+1}$. These intersections are given by
\begin{equation}\label{eq:defineEvoluta}
E_{i+\frac{1}{2}}=P_i-\mu_{i+\frac{1}{2}}U_i=P_{i+1}-\mu_{i+\frac{1}{2}}U_{i+1},
\end{equation}
where $\mu_{i+\frac{1}{2}}$, $1\leq i\leq 2n$, is defined by
\begin{equation}\label{eq:defineCurvature}
P_{i+1}-P_i=\mu_{i+\frac{1}{2}}\left( U_{i+1}-U_i \right).
\end{equation}

\paragraph{Curvature center and radius} In \cite{Petty}, three different notions of Minkowskian curvature are defined, where the circular curvature is directly related to evolutes.
The circular center $E$ and the corresponding radius of curvature $\mu$ are defined by the condition that $E+\mu U$ has a $3$-order contact with the curve at a given point (see \cite{Tabach97}).

For polygons, we define the {\it center of curvature} $E_{i+\frac{1}{2}}$ and the {\it curvature radius} $\mu_{i+\frac{1}{2}}$ of the side $P_iP_{i+1}$ by the condition that
the $(i+\frac{1}{2})$-side of $E_{i+\frac{1}{2}}+\mu_{i+\frac{1}{2}}U$ matches exactly $P_iP_{i+1}$ (see Figure \ref{fig:Curvature}). Thus we get equations \eqref{eq:defineEvoluta} and \eqref{eq:defineCurvature}.
From equations \eqref{eq:defineVlength} and \eqref{eq:defineCurvature} we obtain that
the curvature radius of the side $P_iP_{i+1}$ is also given by
\begin{equation}\label{eq:Curvature2}
\mu_{i+\frac{1}{2}}=\frac{\lambda_{i+\frac{1}{2}}}{[U_i,U_{i+1}]}.
\end{equation}

\begin{figure}[htb]
 \centering
 \includegraphics[width=0.50\linewidth]{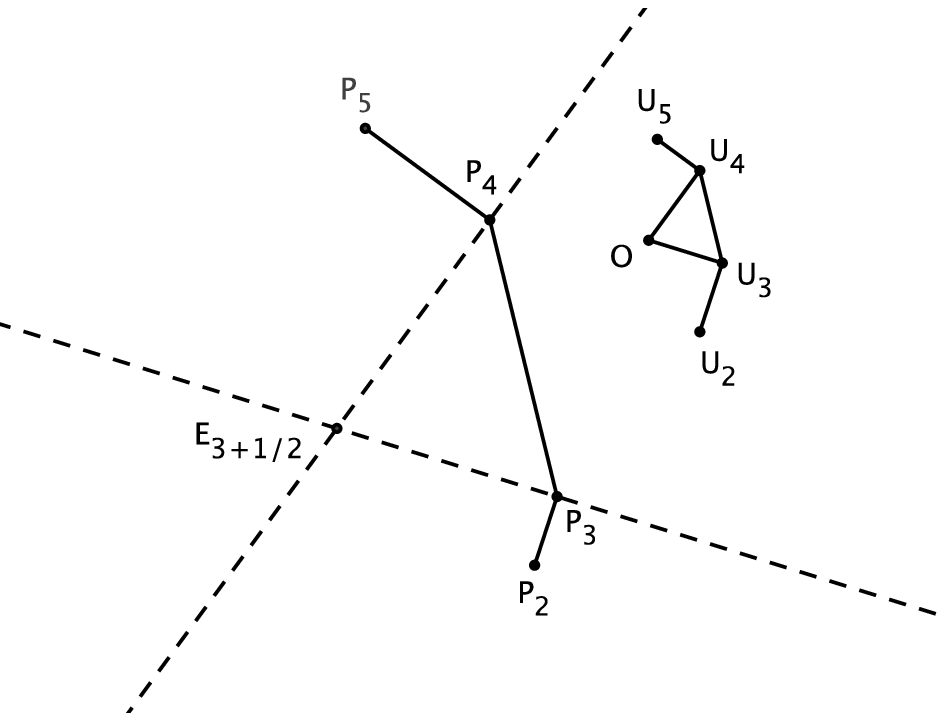}
 \caption{ The center of curvature of the side $P_3P_4$. }
\label{fig:Curvature}
\end{figure}

A vertex $E_{i+\frac{1}{2}}$ is a cusp of the evolute if the vertices $E_{i-\frac{1}{2}}$ and $E_{i+\frac{3}{2}}$ are in the same half-plane defined by the parallel to $P_iP_{i+1}$ through $E_{i+\frac{1}{2}}$.
The evolute of a CW polygon coincides with its \emph{center symmetry set} as defined in \cite{Craizer13}, where 
it is proved that it coincides with the union of cusps of all equidistants of $P$. It is also proved in \cite{Craizer13} that the number of cusps of the evolute is odd and at least the number of cusps of the central equidistant (see Figure \ref{fig:OctoEvolute}).

\begin{figure}[htb]
 \centering
 \includegraphics[width=0.80\linewidth]{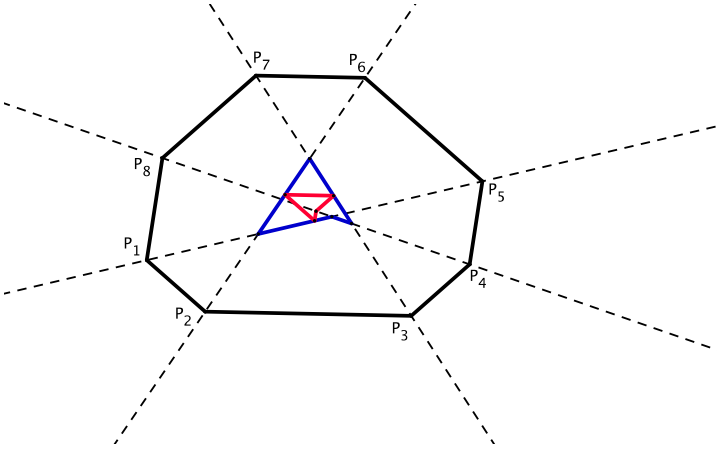}
 \caption{ The inner polygonal arc is the central equidistant $M$ of $P$, and the outer polygonal arc is its evolute $E$. }
\label{fig:OctoEvolute}
\end{figure}

\paragraph{Sum of curvature radii}

Consider equation \eqref{eq:defineCurvature} for two opposite sides, and sum up to obtain, for $1\leq i\leq n$,
\begin{equation*}
P_{i+1}-P_{i+n+1}+P_{i+n}-P_i=(\mu_{i+\frac{1}{2}}+\mu_{i+n+\frac{1}{2}}) (U_{i+1}-U_i).
\end{equation*}
Since $P$ has constant Minkowskian width,
\begin{equation*}
2c(U_{i+1}-U_i)=(\mu_{i+\frac{1}{2}}+\mu_{i+n+\frac{1}{2}}) (U_{i+1}-U_i).
\end{equation*}
We conclude that
\begin{equation}\label{eq:SumReciprocalCurvature}
\mu_{i+\frac{1}{2}}+\mu_{i+n+\frac{1}{2}}=2c.
\end{equation}
The corresponding result for smooth curves is given in \cite{Petty}, Th. 6.14.(c).

\paragraph{Involutes and equidistants }

Consider the one-parameter family of equidistants given by equation \eqref{eq:REescalaEqui}.
The radius of curvature of $P_i(c)P_{i+1}(c)$ is the radius of curvature of $M_iM_{i+1}$ plus $c$. Thus, for $1\leq i\leq 2n$,
\begin{equation}
E_{i+\frac{1}{2}}(c)=M_i+cU_i-\left(\mu_{i+\frac{1}{2}}+c\right)U_i=E_{i+\frac{1}{2}}.
\end{equation}
We conclude that the evolute of any equidistant of $P$ is equal to the evolute of $P$. Reciprocally, any polygonal arc
whose evolute is equal to $E(P)$ is an equidistant of $P$. We define an {\it involute} of $E$ as any polygonal arc whose evolute is $E$.
Thus the involutes of $E$ are the equidistants of $P$.

\subsection{The signed area of the central equidistant} \label{sec:SignedAreas}

Given two closed curves $P$ and $Q$, the mixed area of their convex hulls is defined by the equation
$$
A(P+tQ)=A(P)+2tA(P,Q)+t^2A(Q).
$$
The Minkowski inequality says that $A(P,Q)^2\geq A(P)A(Q)$. The next lemma is well-known, see \cite[\S\S~6.3]{Gruber}.

\begin{lemma}\label{lemma:MixedArea}
Take $P$ and $Q$ as convex polygons with $k$ parallel corresponding sides. The {\it mixed area} of $P$ and $Q$ is given by
\begin{equation*}
A(P,Q)=\frac{1}{2}\sum_{i=1}^{k}[Q_{i},P_{i+1}-P_i]=\frac{1}{2}\sum_{i=1}^{k}[P_{i+1}, Q_{i+1}-Q_i].
\end{equation*}
\end{lemma}

Assume that $P$ is a closed convex polygon whose sides are parallel to the sides of the centered polygon $U$, and take $Q=U$ in Lemma \ref{lemma:MixedArea}. We obtain
$$
A(P,U)=\frac{1}{2}\sum_{i=1}^{2n} [U_{i},P_{i+1}-P_i]=\frac{1}{2}\sum_{i=1}^{2n}\lambda_{i+\frac{1}{2}}=\frac{1}{2}L_V(P),
$$
where we have used \eqref{eq:defineVlength} and \eqref{eq:defineVlength2}. Moreover, the Minkowski inequality becomes
\begin{equation}\label{eq:Isoperimetric}
L^2_V(P)\geq 4A(U)A(P).
\end{equation}

\begin{lemma}
Let $M$ be the central equidistant of a CW-polygon $P$. Then the mixed area $A(M,M)$ is non-positive.
\end{lemma}

\begin{proof}
Let $P(c)$ be defined by equation \eqref{eq:REescalaEqui}. Then
$$
A(P(c),P(c))=A(M,M)+2cA(M,U)+c^2A(U,U).
$$
Now equation \eqref{eq:MlengthZero} says that $A(M,U)=0$. Moreover, the isoperimetric inequality \eqref{eq:Isoperimetric} for curves of constant width says that
$$
A(P)\leq c^2A(U).
$$
We conclude that
$$
A(M,M)\leq 0.
$$
\end{proof}

Define the {\it signed area} of $M$ as $SA(M)=-A(M,M)$. In general, the signed area
is a sum of positive and negative areas, but when $M$ is a simple curve, it coincides with the area bounded by $M$.

\subsection{Relation between length and area of a half polygon}

Define $\beta_i$ by
\begin{equation}\label{eq:defineBeta}
\beta_i=\frac{1}{2}\sum_{j=i}^{n+i-1}\alpha_{j+\frac{1}{2}}[U_j,U_{j+1}].
\end{equation}
Observe that $\beta_{i+n}=-\beta_i$, $1\le i\leq n$, and
\begin{equation}\label{eq:deriveBeta}
 \beta_{i+1}-\beta_i =-\alpha_{i+\frac{1}{2}}[U_i,U_{i+1}].
\end{equation}

Denote by $A_{1}(i,c)$ and $A_2(i,c)$ the areas of the polygons with vertices \linebreak 
$\{P_i,P_{i+1},...,P_{i+n}\}$ and $\{P_{i+n},P_{i+n+1},...,P_{i}\}$. Observe that 
these polygons are bounded by $P$ and the diagonal $P_iP_{i+n}$.

\begin{Proposition}\label{prop:AreaProperty}
We have that
$$
A_1(i,c)-A_2(i,c)=4c\beta_i,
$$
for $1\leq i\leq 2n$.
\end{Proposition}

\begin{proof}
Lemma 4.1. of \cite{Craizer13} says that
$$
A_1(i,c)-A_2(i,c)=-2\sum_{j=i}^{i+n-1} [M_{j+1}-M_j, cU_j]
$$
$$
=-2c\sum_{j=i}^{i+n-1}[\alpha_{j+\frac{1}{2}}[U_j,U_{j+1}]V_{j+\frac{1}{2}},U_j].
$$
Thus
$$
A_1(i,c)-A_2(i,c)=2c\sum_{j=i}^{i+n-1}\alpha_{j+\frac{1}{2}}[U_j,U_{j+1}]=4c\beta_i.
$$
\end{proof}

Denote by $L_V(i,c)$ the $V$-length of the polygonal arc whose vertices are $\{P_i(c),P_{i+1}(c),...,P_{i+n}(c)\}$. Then
\begin{equation}\label{eq:defineLi}
L_V(i,c)=\sum_{j=i}^{i+n-1}(\alpha_{i+\frac{1}{2}}+c)[U_j,U_{j+1}]=2cA(U)+2\beta_i.
\end{equation}

\begin{corollary} For $1\leq i\leq 2n$, the expression
$A_1(i,c)-cL_V(i,c)$ is independent of $i$.
\end{corollary}
\begin{proof}
By equation \eqref{eq:defineLi} and Proposition \ref{prop:AreaProperty}, we get
$$
2cL_V(i,c)-2A_1(i,c)=4c^2A(U)+4c\beta_i-2A_1(i,c)=4c^2A(U)-A(P),
$$
which proves the corollary.
\end{proof}

The above corollary presents the ``polygonal analogue'' of a known theorem holding for strictly convex curves (see \cite{Chakerian83}, eq. (2.1)).

\section{The involute of the central equidistant}

Recall that $P=\{P_1,...,P_{2n}\}$ is a convex polygon with parallel opposite sides and $U=\{U_1,...,U_{2n}\}$ is the Minkowski ball obtained
from $P$ by the construction of  Lemma \ref{lemma:SymmetricBall}. The polygon $V=\{V_1,...,V_{2n}\}$ represents the dual Minkowski ball (see Lemma \ref{lemma:DualBall})
and $M=\{M_1,...,M_n\}$ is the central equidistant of $P$ (see equation \eqref{eq:defineCentral}).

\subsection{Basic properties of the involute $N$ of $M$}

Define the polygon $N$ by
\begin{equation}\label{eq:defineInvoluta}
 N_{i+\frac{1}{2}}=M_i+\beta_i V_{i+\frac{1}{2}},
\end{equation}
$1\leq i\leq 2n$.
Observe that $N_{i+\frac{1}{2}}=N_{i+n+\frac{1}{2}}$. Due to equations \eqref{eq:defineAlpha} and \eqref{eq:deriveBeta}, we can also write
\begin{equation}\label{eq:defineInvoluta2}
N_{i+\frac{1}{2}}=M_{i+1}+\beta_{i+1} V_{i+\frac{1}{2}}.
\end{equation}

\begin{lemma}
The polygon $N$ has constant $V$-width, and the evolute of $N$ is $M$.
\end{lemma}
\begin{proof}
Since
\begin{equation}
N_{i+\frac{1}{2}}-N_{i-\frac{1}{2}}=\beta_i \left( V_{i+\frac{1}{2}}-  V_{i-\frac{1}{2}} \right),
\end{equation}
$1\leq i\leq n$, the sides of $N$ are parallel to the sides of $V$. Moreover, the diagonals of $N$ are zero, so they are multiples of the diagonals of $V$. We conclude from
Corollary \ref{cor:CWequivalence} that $N$ has constant $V$-width. Finally, from equation \eqref{eq:defineInvoluta} we conclude that the evolute of $N$ is $M$.
\end{proof}

The equidistants of $N$, which are the involutes of $M$, are curves of constant $V$-width
(see Figure \ref{fig:OctoInvolute}). In \cite{Craizer13}, these polygons were called the Parallel Diagonal Transforms of $P$.

\begin{figure}[htb]
 \centering
 \includegraphics[width=0.80\linewidth]{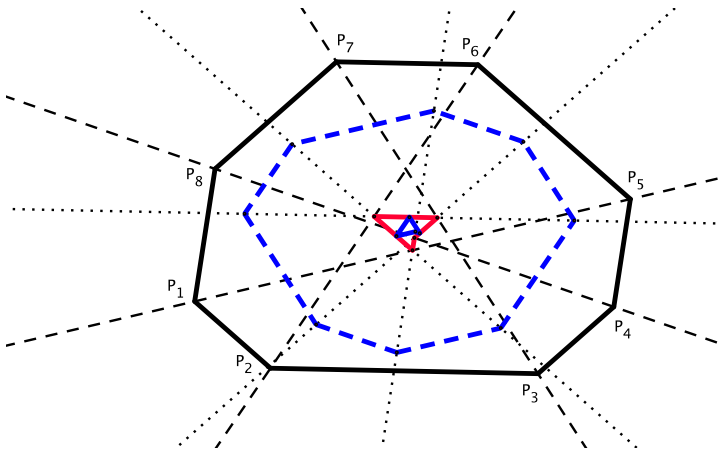}
 \caption{  The central equidistant $M$ together with two involutes of $M$: The inner curve is the central equidistant $N$, and the traced curve is an ordinary involute. }
\label{fig:OctoInvolute}
\end{figure}

\subsection{ The signed area of the involute of the central equidistant}

For smooth convex curves of constant Minkowskian width, the signed area of $N$ is not larger than the signed area of $M$ (see \cite{Craizer14}). We prove here the corresponding result for polygons.

\begin{Proposition}\label{prop:signedMN}
Denoting by $SA(M)$ and $SA(N)$ the signed areas of $M$ and $N$, we have
$$
SA(M)-SA(N)=\sum_{i=1}^{n}\beta_i^2 \left[ V_{i-\frac{1}{2}},V_{i+\frac{1}{2}} \right].
$$
\end{Proposition}
\begin{proof}
Observe that
$$
\left[ M_i,M_{i+1}  \right]=  \left[ N_{i+\frac{1}{2}}-\beta_i V_{i+\frac{1}{2}}, \alpha_{i+\frac{1}{2}} (U_{i+1}-U_i )\right]= \alpha_{i+\frac{1}{2}}[N_{i+\frac{1}{2}},U_{i+1}-U_i]=
$$
$$
-(\beta_{i+1}-\beta_i)[N_{i+\frac{1}{2}},V_{i+\frac{1}{2}} ],\ \ \left[  N_{i-\frac{1}{2}},N_{i+\frac{1}{2}}\right] =\beta_i \left[ N_{i+\frac{1}{2}}, V_{i+\frac{1}{2}}-V_{i-\frac{1}{2}} \right],
$$
and so
$$
-\left[ M_i,M_{i+1} \right]+\left[  N_{i-\frac{1}{2}},N_{i+\frac{1}{2}} \right] =[N_{i+\frac{1}{2}}, \beta_{i+1}V_{i+\frac{1}{2}}-   \beta_{i}V_{i-\frac{1}{2}}  ] .
$$
Thus
$$
SA(M)-SA(N)= \sum_{i=1}^{n} - \left[ M_i,M_{i+1} \right]+\left[  N_{i-\frac{1}{2}} ,N_{i+\frac{1}{2}} \right] =
$$
$$
=-\sum_{i=1}^{n} \left[ N_{i+\frac{1}{2}}-N_{i-\frac{1}{2}}, \beta_iV_{i-\frac{1}{2}} \right]=\sum_{i=1}^{n}\beta_i^2 \left[ V_{i-\frac{1}{2}},V_{i+\frac{1}{2}} \right],
$$
where we have used that the difference
$$
[N_{i+\frac{1}{2}},\beta_{i+1}V_{i+\frac{1}{2}}]-[N_{i-\frac{1}{2}},\beta_{i}V_{i-\frac{1}{2}}]
$$
is equal to
$$
[N_{i+\frac{1}{2}}-N_{i-\frac{1}{2}},\beta_iV_{i-\frac{1}{2}}]+[N_{i+\frac{1}{2}},\beta_{i+1}V_{i+\frac{1}{2}}-\beta_{i}V_{i-\frac{1}{2}}],
$$
the discrete version of "integration by parts".
\end{proof}

\subsection{The involute is contained in the interior of the central equidistant}

We prove now that the region bounded by the central equidistant $M$ contains its involute $N$.  For smooth convex curves, this result was proved in \cite{Craizer14}.

The exterior of the curve $M$ is defined as the set of points of the plane that can be reached from a point of $P$ by a path that does not cross $M$.
The region $\overline{M}$ bounded by $M$ is the complement of its exterior. It is well known that a point in the exterior of $M$ is the center of exactly one chord of $P$ (see \cite{Craizer13}).

\begin{Proposition}\label{prop:NsubsetM}
The involute $N$ is contained in the region $\overline{M}$ bounded by $M$.
\end{Proposition}

The proof is based on two lemmas. For a fixed index $i$, denote by $l(i)$ the line parallel to $P_{i+n}-P_i$ through $N_{i-\frac{1}{2}}$ and $N_{i+\frac{1}{2}}$. Then $l(i)$ divides the interior of $P$ into two regions of areas $B_1=B_1(i)$ and $B_2=B_2(i)$, where the second one contains $P_i$ and $P_{i+n}$.

\begin{lemma}
We have that $B_1(i)\geq B_2(i)$, $1\leq i\leq n$.
\end{lemma}
\begin{proof}
We have that
$$
B_1(i)=A_1(i)-(2c\beta_i-\delta_i-\eta_i),\ \ B_2(i)=A_2(i)+(2c\beta_i-\delta_i-\eta_i),
$$
where $\delta_i$ is the area of the regions outside $P$ and between $l(i)$, $P_iP_{i+n}$ and the support lines of $P_iP_{i+1}$ and $P_{i+n-1}P_{i+n}$, and $\eta_i$ is the area of the triangle
$M_iN_{i+\frac{1}{2}}N_{i-\frac{1}{2}}$ (see Figure \ref{fig:NAreas1}). Since,
by Proposition \ref{prop:AreaProperty}, $4c\beta_i=A_1-A_2$, we conclude that
$$
B_1(i)=\frac{A(P)}{2}+\delta_i+\eta_i,\ \ B_2(i)=\frac{A(P)}{2}-\delta_i-\eta_i,
$$
which proves the lemma.
\end{proof}

\begin{figure}[htb]
 \centering
 \includegraphics[width=0.70\linewidth]{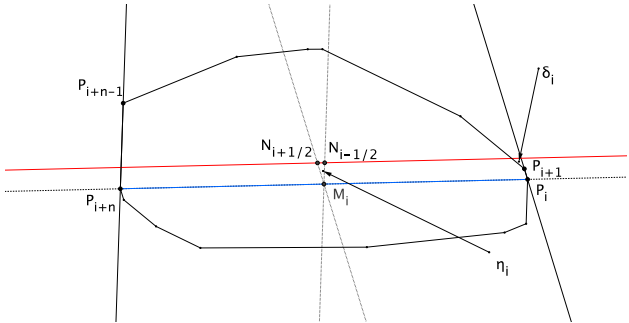}
 \caption{ The line through $N_{i+\frac{1}{2}}$ and $N_{i-\frac{1}{2}}$ divides the polygon into two regions of areas $B_1$ and $B_2$. }
\label{fig:NAreas1}
\end{figure}

\begin{lemma}\label{lemma:interiorM}
Choose $C$ in the segment $N_{i-\frac{1}{2}}N_{i+\frac{1}{2}}$. Then $C$ is in the region bounded by $M$.
\end{lemma}
\begin{proof}
By an affine transformation of the plane, we may assume that $l(i)$ and $M_iC$ are orthogonal.
Consider polar coordinates $(r,\phi)$ with center $C$ and describe $P$ by $r(\phi)$. Assume that $\phi=0$ at the line $l(i)$ and that $\phi=-\phi_0$ at $P_i$.
Denote the area of the sector bounded by $P$ and the rays $\phi_1,\phi_2$ by
$$
A(\phi_1,\phi_2)=\frac{1}{2}\int_{\phi_1}^{\phi_2}r^2(\phi)d\phi.
$$

Consider a line parallel to $M_iC$ and passing through the point  $Q_0$ of $P$ corresponding to $\phi=0$, and denote by $Q_1$ and $Q_2$ its intersection with the rays $\phi=-\phi_0$
and $\phi=\phi_0$, respectively (see Figure \ref{fig:NAreas2}). By convexity, we have that
$$
A(0,\phi_0)\leq A(CQ_0Q_1)=A(CQ_0Q_2)\leq A(-\phi_0,0).
$$
A similar reasoning shows that
$ A(\pi-\phi_0,\pi)\leq A(\pi,\pi+\phi_0)$. Observe also that, by convexity, $r(\phi_0)\leq r(\phi_0+\pi)$ and $r(\pi-\phi_0)\leq r(-\phi_0)$.

Now, if $r(\phi+\pi)>r(\phi)$ for any $\phi_0<\phi<\pi-\phi_0$, we would have $B_1(C)<B_2(C)$, contradicting the previous lemma. We conclude that
$r(\phi+\pi)=r(\phi)$ for at least two values of $\phi_0<\phi<\pi-\phi_0$. Since equality holds also for some $\pi-\phi_0<\phi<\pi+\phi_0$,
there are at least three chords of $\gamma$ having $C$ as midpoint. Thus $C$ is contained in the region bounded by $M$.
\end{proof}

\begin{figure}[htb]
 \centering
 \includegraphics[width=0.60\linewidth]{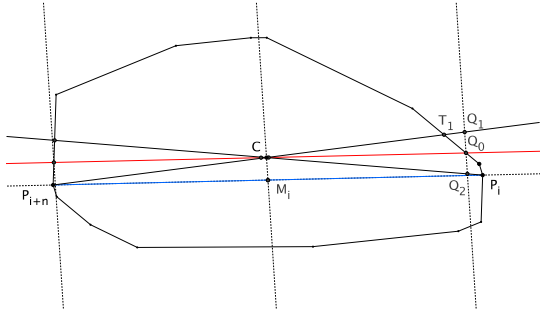}
 \caption{ The line parallel to $M_iC$ through $Q_0$ determines the points $Q_1$ and $Q_2$. }
\label{fig:NAreas2}
\end{figure}

We can now complete the proof of Proposition \ref{prop:NsubsetM}. In fact, from Lemma \ref{lemma:interiorM} we have that each side $N_{i-\frac{1}{2}}N_{i+\frac{1}{2}}$ is contained in the region
$\overline {M}$
bounded by $M$. Therefore, no point on the boundary of $N$ can be connected with the boundary of $P$ by a curve that does not intersect $M$. This implies that the region $\overline{N}$ bounded by $N$
is contained in $\overline{M}$.

\section{Iterating involutes}

Starting with the central equidistant $M=M(0)$ and its involute $N=N(1)$, we can iterate the involute operation. We obtain two sequences of $n$-gons $M(k)$ and $N(k)$ defined by
$M(k)={\mathcal Inv}(N(k))$ and $N(k+1)={\mathcal Inv}(M(k))$. For smooth curves of constant Minkowskian width, it is proved in \cite{Craizer14} that these sequences converge to a constant. We
prove here the corresponding result for polygons.

From Proposition \ref{prop:NsubsetM}, we have
$$
\overline {M(0)}\supset\overline {N(1)}\supset\overline {M(1)}\supset ...,
$$
and we denote by $O=O(P)$ the intersection of all these sets.

If we represent a polygon by its vertices, we can embed the space ${\mathcal P}_{n}$ of all $n$-gons in $(\R^2)^{n}$. In ${\mathcal P}_{n}$ we consider the topology induced
by $\R^{2n}$.

\begin{thm}\label{thm:ConvMiNi}
The set $O=O(P)$ consists of a unique point, and the polygons $M(k)$ and $N(k)$ are converging to $O$ in ${\mathcal P}_{n}$.
\end{thm}

We shall call $O=O(P)$ the {\it central point} of $P$. A natural question that arises is the following.
\paragraph{Question} Is there a direct method to obtain the central point $O$ from the polygon $P$?

For fixed $c$ and $d$ construct the sequences of convex polygons $P(k,c)$ and $Q(k,d)$ whose vertices are
$$
P_i(k)=M_i(k)+cU_i(k), \ \  Q_{i+\frac{1}{2}}(k)=N_{i+\frac{1}{2}}(k)+dV_{i+\frac{1}{2}}(k)\,,
$$
respectively.
The polygons $P(k,c)$ are of constant $U$-width, while the polygons $Q_{i+\frac{1}{2}}(k,d)$ are of constant $V$-width.
We can re-state Theorem \ref{thm:ConvMiNi} as follows:

\begin{thm}
The sequences of polygons $P(k,c)$ and $Q(k,d)$ are converging in ${\mathcal P}_{2n}$ to $O+c\partial\U$ and $O+d\partial\V$, respectively.
\end{thm}

\begin{figure}[htb]
 \centering
 \includegraphics[width=0.90\linewidth]{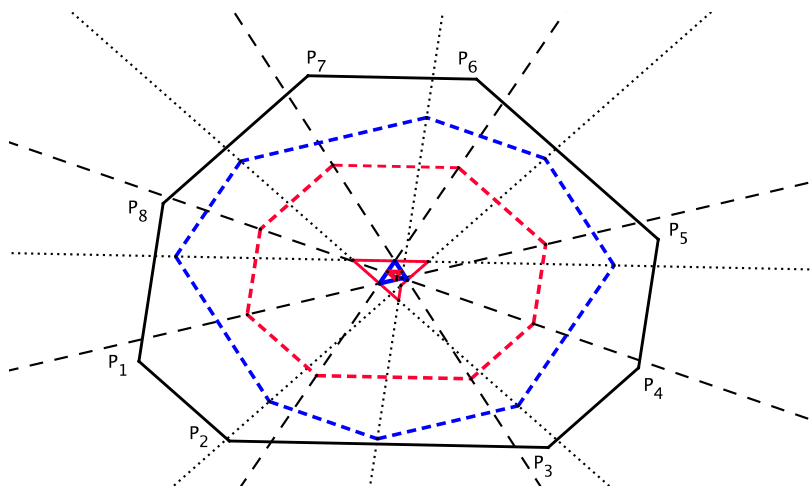}
 \caption{ The inner curves are $M=M(0)$, $N=N(1)$ and $M(1)$. One traced curve is an ordinary $V$-equidistant of $N$, and the other one is
 an ordinary $U$-equidistant of $M(1)$.  }
\label{fig:OctoIterates}
\end{figure}

\bigskip

\bigskip\noindent
We shall prove now Theorem \ref{thm:ConvMiNi}.

\begin{proof}
Denote the signed areas of $M(k)$ and $N(k)$ by $SA(M(k))$ and $SA(N(k))$, respectively.
By Section \ref{sec:SignedAreas}, $SA(M(k))\geq 0$, $SA(N(k))\geq 0$, and Proposition \ref{prop:signedMN} implies that
$$
SA(M(k))-SA(N(k+1))=\sum_{i=1}^{n} \beta_{i}^2(k)[U_i,U_{i+1}],
$$
$$
SA(N(k))-SA(M(k))=\sum_{i=1}^{n} \alpha_{i+\frac{1}{2}}^2(k)[V_{i-\frac{1}{2}},V_{i+\frac{1}{2}}],
$$
where $\alpha_{i+\frac{1}{2}}(k)$ and $\beta_i(k)$ are defined by
\begin{equation*}
\begin{array}{l}
M_{i+1}(k)-M_i(k)=\alpha_{i+\frac{1}{2}}(k)( U_{i+1}-U_i),\ \\
N_{i+\frac{1}{2}}(k)-N_{i-\frac{1}{2}}(k)=\beta_{i}(k) (V_{i+\frac{1}{2}}-V_{i-\frac{1}{2}}).
\end{array}
\end{equation*}
We conclude that
\begin{equation}\label{eq:sumsquares}
\sum_{k=1}^{\infty} \sum_{i=1}^{n} \beta_{i}^2(k)[U_i,U_{i+1}] +\sum_{k=0}^{\infty} \sum_{i=1}^{n} \alpha_{i+\frac{1}{2}}^2(k)[V_{i-\frac{1}{2}},V_{i+\frac{1}{2}}] \leq SA(M(0)).
\end{equation}

From the above equation, we obtain that the sequences $\alpha_{i+\frac{1}{2}}(k)$ and $\beta_i(k)$ are converging to $0$ in $\R^{n}$.
So the diameters of $M(k)$ and $N(k)$ are converging to zero, and thus $O$ is in fact a set consisting of a unique point.
\end{proof}

\end{document}